\documentclass[letterpaper, 10pt, conference]{ieeeconf}
\IEEEoverridecommandlockouts
\usepackage{amsmath,amssymb,amsfonts,amsthm}
\usepackage{graphicx}
\usepackage{cite}

\newtheorem{theorem}{Theorem}
\newtheorem{lemma}{Lemma}
\newtheorem{proposition}{Proposition}
\newtheorem{corollary}{Corollary}
\newtheorem{definition}{Definition}
\newtheorem{remark}{Remark}

\newcommand{\R}{\mathbb{R}}
\newcommand{\eps}{\varepsilon}

\title{\LARGE \bf Fundamental Limits of Prescribed-Time Regulation:\\ Gain Budgets, Accuracy Frontiers, and Arrival Costs}

\author{Osama~F.~Abdel~Aal, Jiale~Chen, Jairo~Viola, and YangQuan~Chen$^{\dagger}$%
\thanks{Osama F. Abdel Aal, Jiale Chen, Jairo Viola, and YangQuan Chen are with the Department of Mechanical and Aerospace Engineering, University of California, Merced, CA 95343, USA \texttt{\{oabdelaal,jialechen2,jviola,ychen53\}@ucmerced.edu}}%
\thanks{Jiale Chen is also with the School of Mathematical Sciences, Hangzhou Dianzi University, Hangzhou 310018, China.}%
\thanks{$^{\dagger}$Deceased. See the Acknowledgment.}}

\begin{document}
\maketitle
\thispagestyle{empty}
\pagestyle{empty}

\begin{abstract}
Prescribed-time controllers drive the state to the origin at a time $T$ chosen in advance, with gains that grow without bound as $t$ approaches $T$. This paper shows the growth is a necessity and computes the exact price of avoiding it. The starting point is a simple estimate: any trajectory with $\|\dot z(t)\|\le L(t)\|z(t)\|$ satisfies $\|z(t)\|\ge\|z(0)\|e^{-\int_0^t L}$. From this: a gain profile $L$ supports exact prescribed-time convergence, over the unrestricted class of closed-loop vector fields admitting it, if and only if its integral over $[0,T)$ is infinite; so every design must either use an envelope of infinite integral or violate the linear growth bound at the origin, as finite-time and fixed-time designs do. Under a finite gain budget $\mathcal B=\int_0^T L$ the best worst-case contraction over the class is exactly $e^{-\mathcal B}$, a floor valid for every plant, and a finite-budget trajectory converges to a definite point away from the origin. A comparison lemma then carries the theory beyond linear growth: for envelopes $\|\dot z\|\le L(t)\,\omega(\|z\|)$, first arrival at the origin requires budget at least $\int_0^{\|z_0\|}du/\omega(u)$; arrival is impossible whenever this integral diverges, and when it is finite it is the exact threshold, attained inside the class. The classical settling expressions of finite-time control reappear as hitting-time floors, and the known fixed-time condition is the two-ended convergence of the same integral. For linear time-varying loops the results refine to logarithmic-norm budgets, corollaries of a known two-sided estimate, with the terminal error of gain freezing described exactly.
\end{abstract}


\section{Introduction}

Prescribed-time (PT) control asks for convergence to the origin at a time $T$ fixed in advance by the designer, uniformly over initial conditions. Since the time-varying-gain design of Song, Wang, Holloway and Krsti\'c \cite{SongKrstic2017}, the area has grown quickly: observers \cite{HollowayKrstic2019,Espitia2022}, dynamic high-gain designs \cite{Krishnamurthy2020}, linear time-varying feedback for nonlinear systems \cite{ZhouShi2021,ZhangZhouDuan2024}, time-space deformation \cite{Orlov2022}, differentiators \cite{OrlovCSL2022}, decentralized schemes \cite{Ye2024}, prescribed-performance tracking \cite{LiMaTian2024}; see the survey \cite{SongSurvey2023}. These designs share a signature: a gain growing like $(T-t)^{-1}$ or faster as $t\to T$.

Whether this growth is truly necessary is debated: gain-freezing implementations stop it shortly before $T$ and accept a terminal error \cite{Shakouri2022,Zhou2020}, while recent work announces PT convergence with bounded gain \cite{SongSu2023}. By comparison, converse results are scarce; we are aware of two lines. Aldana-L\'opez et al.\ \cite{Aldana2023} proved that a class of PT algorithms is fragile under measurement noise, a direction orthogonal to ours. Closest to this paper, the technical communiqu\'e of Shakouri \cite{Shakouri2022} proved, for linear time-varying (LTV) systems under the hypothesis that every entry of $A(t)$ has a limit as $t\to T^-$, that PT attractivity forces $\|A(t)\|\to\infty$; the note also contains a sufficient condition, an exact non-constructive characterization, and the gain-freezing implementation. Its Lemma~1 is the two-sided logarithmic-norm estimate of Vr\'abel \cite{Vrabel2020}, with roots in \cite{Coppel1965,Strom1975,Soderlind2006}; our LTV statements in Section~\ref{sec:ltv} are corollaries of that estimate which \cite{Shakouri2022} did not state: the integral form of the necessity, free of the entrywise-limit hypothesis. The main contributions concern a different question: not whether a given linear system is PT attractive, but what \emph{any} closed loop, of any structure, can achieve under a stated gain constraint.

In this paper, we characterize exactly which gain envelopes support exact PT convergence over the unrestricted class of closed-loop vector fields admitting them: a locally integrable envelope works if and only if $\int_0^T L=\infty$ (Theorem \ref{thm:char}), with necessity a lower Gr\"onwall bound along every trajectory of every plant and sufficiency from a design inside the class, so that no envelope bounded by $C(T-t)^{-\theta}$ with $\theta<1$ can work and, among power-law envelopes, the exponent $\theta=1$ is the smallest admissible. We then compute the exact accuracy frontier under a finite budget $\mathcal B=\int_0^TL$: the smallest worst-case terminal contraction over the class equals $e^{-\mathcal B}$ (Theorem \ref{thm:frontier}), a floor binding every trajectory of every plant, attained at class level, while its realization on a specified plant is a separate question (Remark \ref{rem:plant}). We carry the floors beyond linear growth: for envelopes $\|\dot z\|\le L(t)\,\omega(\|z\|)$, one comparison lemma assigns every modulus the arrival cost $\mathcal C(r)=\int_0^{r}du/\omega(u)$ and yields first-arrival floors for every admissible trajectory in every dimension (Section \ref{sec:osgood}); divergence of this integral at the origin makes arrival impossible under any finite budget, the classical settling expressions of finite-time control \cite{Haimo1986,BhatBernstein2000,MoulayPerruquetti2008} reappear as floors attained by the inward radial field, and the fixed-time condition of \cite{Polyakov2012} is the two-ended convergence of the same integral. Finally, for LTV loops we give the integral and essential-supremum forms of the singularity necessity, free of extra assumptions, as corollaries of the known two-sided estimate, the complete scalar and diagonal characterization, and the exact terminal-error spectrum of gain freezing, with our budget floor beneath it (Section \ref{sec:ltv}).

We see these results the way converse bounds usually function: not as an obstacle to the constructive literature \cite{SongKrstic2017,Krishnamurthy2020,ZhouShi2021,Orlov2022,Ye2024}, but as the limits that come with it.

\emph{Notation.} $\|\cdot\|$ is the Euclidean norm on vectors and the induced norm on matrices, $\|\cdot\|_p$ the $p$-norm family, and $|\cdot|$ the absolute value of a scalar. For measurable $g$, $g^+=\max(g,0)$. AC means absolutely continuous; a.e.\ refers to Lebesgue measure. $\sigma_{\min},\sigma_{\max}$ are extreme singular values. The logarithmic norm is $\mu_p(M)=\lim_{h\to0^+}(\|I+hM\|_p-1)/h$; it is Lipschitz in $M$ with $|\mu_p(M)|\le\|M\|_p$ \cite{Strom1975,Soderlind2006}.

\section{Setting}\label{sec:setting}

\begin{definition}\label{def:env}
Let $z:[0,T)\to\R^n$ be AC. A measurable $L:[0,T)\to[0,\infty)$ is a \emph{gain envelope} for $z$ if
\begin{equation}\label{eq:env}
\|\dot z(t)\|\ \le\ L(t)\,\|z(t)\|\qquad\text{a.e. }t\in[0,T).
\end{equation}
A closed loop $\dot z=f(t,z)$, with $f$ Carath\'eodory, admits the envelope if $\|f(t,z)\|\le L(t)\|z\|$ for all $z$ and a.e.\ $t$. For $\dot z=(A_0+BK(t))z$ one may take $L=\|A_0+BK(\cdot)\|$. The \emph{budget} is $\mathcal B=\int_0^TL(t)\,dt$, and $\bar L$ denotes a \emph{peak gain}, a constant with $L(t)\le\bar L$ (so $\mathcal B\le\bar LT$).
\end{definition}

Condition \eqref{eq:env} is a linear \emph{growth} bound on the field, not linearity: the class is genuinely nonlinear, containing $f(t,z)=-L(t)z\sin^2\|z\|$, saturated feedback $-L(t)\mathrm{sat}(z)$, state-dependent gains $-k(t,z)z$ with $0\le k\le L$, and switching among such fields. A field can fail to admit a global envelope in more than one way; the failure relevant to time regulation is an unbounded ratio $\|f(t,z)\|/\|z\|$ near the origin, as for the fractional-power fields $-\|z\|^{\alpha-1}z$, $\alpha\in(0,1)$; this failure is route (b) of Corollary \ref{cor:routes}, and Section \ref{sec:osgood} covers that side. The bound \eqref{eq:env} neither implies nor is implied by Lipschitz continuity of $f$ away from the origin, and no proof below uses any regularity of $f$ beyond \eqref{eq:env} along the trajectory at hand, so static, dynamic, output, switching and hybrid loops are covered uniformly. The linearity of the attaining designs is a conclusion, not an assumption: within the class, nonlinearity cannot improve on linear feedback. The choice of norm is free: the floors hold in every norm with the same proofs, the envelope and budget measured in that norm.

\begin{remark}[Budget versus input effort]\label{rem:effort}
$\mathcal B$ is the integral of the \emph{envelope}, not of the input, and the two must not be identified. The scalar loop $\dot x=u$ with $u=-x/(T-t)$ has solution $x(t)=x_0(T-t)/T$: its budget is infinite, while its input is the finite constant $u(t)=-x_0/T$. Budget floors are statements about gain accumulation; limits in input energy are a separate question.
\end{remark}

Given $T>0$ and an AC trajectory $z$ on $[0,T)$ with $z(0)=z_0\ne0$: $z$ is \emph{exactly PT-regulated at $T$} if $z(t)\to0$ as $t\to T^-$; \emph{$\eps$-accurate at $T$} if $\liminf_{t\to T^-}\|z(t)\|\le\eps\|z_0\|$; its \emph{first arrival time} is $t_{\mathrm{hit}}=\inf\{t: z(t)=0\}$ ($\infty$ if the state never vanishes). A system is \emph{PT attractive} at $T$ if every solution is exactly PT-regulated \cite{Shakouri2022}. Results are stated in terms of first arrival; reading arrival as settling additionally requires origin invariance after arrival, which the speed bound alone does not guarantee under non-uniqueness, and which the attaining designs below provide.

\section{Gain budgets and the accuracy frontier}\label{sec:main}

\begin{lemma}\label{lem:gronwall}
Let $z:[0,\tau]\to\R^n$ be AC with $z(0)=z_0\ne0$, and let $L\in L^1(0,\tau)$ be a gain envelope for $z$. Then $z(t)\ne0$ on $[0,\tau]$ and
\begin{equation}\label{eq:lower}
\|z(t)\|\ \ge\ \|z_0\|\exp\Big(-\!\int_0^t L(s)\,ds\Big),\qquad t\in[0,\tau].
\end{equation}
\end{lemma}

\begin{proof}
Let $t^*=\inf\{t\in[0,\tau]:z(t)=0\}$, with $t^*=+\infty$ if there is no such time; continuity and $z_0\ne0$ give $t^*>0$. Fix $\tau'<\min(t^*,\tau]$. On $[0,\tau']$ the continuous function $\|z\|$ has a positive minimum, and it is AC with $|\frac{d}{dt}\|z\||\le\|\dot z\|$ a.e., because $\big|\|z(t)\|-\|z(s)\|\big|\le\|z(t)-z(s)\|$. Hence $\varphi=\log\|z\|$ is AC on $[0,\tau']$ with $\dot\varphi\ge-\|\dot z\|/\|z\|\ge-L$ a.e., and integrating gives \eqref{eq:lower} on $[0,\tau']$. If $t^*\le\tau$, let $\tau'\uparrow t^*$: continuity of $z$ gives $\|z(t^*)\|\ge\|z_0\|e^{-\int_0^{t^*}L}>0$ because $L$ is integrable, contradicting $z(t^*)=0$. So $t^*>\tau$ and \eqref{eq:lower} holds on $[0,\tau]$.
\end{proof}

\begin{theorem}\label{thm:necessity}
Let $z$ be an AC trajectory on $[0,T)$ with $z(0)=z_0\ne0$ and gain envelope $L\in L^1_{\mathrm{loc}}([0,T))$.
(i) If $\int_0^T L<\infty$, then $\inf_{t\in[0,T)}\|z(t)\|\ge\|z_0\|e^{-\int_0^T L}>0$; in particular $z$ is not exactly PT-regulated at $T$ or earlier, and $t_{\mathrm{hit}}=\infty$.
(ii) If $z$ is $\eps$-accurate at $T$ with $\eps\in(0,1)$, then $\int_0^T L\ge\log(1/\eps)$.
(iii) If $L\le\bar L$, then $\|z(t)\|\ge e^{-\bar LT}\|z_0\|$ on $[0,T)$.
\end{theorem}

\begin{proof}
(i) follows from Lemma \ref{lem:gronwall} on $[0,\tau]$ for each $\tau<T$. (ii) combines (i) with the definition. (iii) is (i) with $\int_0^TL\le\bar LT$.
\end{proof}

\begin{theorem}\label{thm:char}
Let $L\in L^1_{\mathrm{loc}}([0,T);[0,\infty))$. There exists a closed-loop vector field admitting the envelope $L$ that is PT attractive at $T$ if and only if $\int_0^t L\to\infty$ as $t\to T^-$.
\end{theorem}

\begin{proof}
Necessity is Theorem \ref{thm:necessity}(i). For sufficiency take $f(t,z)=-L(t)z$: since $L$ is locally integrable, the unique AC solution from any $z_0$ is $z(t)=z_0e^{-\int_0^tL}$, the envelope holds with equality, and the integral condition gives $z(t)\to0$.
\end{proof}

\begin{remark}
The condition characterizes the existence of a design within the class, not the behavior of every design in it: $\dot x=+L(t)x$ has the same envelope and diverges.
\end{remark}

\begin{corollary}\label{cor:order}
No closed loop with envelope $L(t)\le C(T-t)^{-\theta}$, $C>0$, $\theta\in(0,1)$, is PT attractive at $T$, although the comparison function $C(T-t)^{-\theta}$ diverges pointwise, since $\int_0^T C(T-s)^{-\theta}ds=CT^{1-\theta}/(1-\theta)<\infty$. Among power-law envelopes, $\theta=1$ is therefore the smallest admissible exponent, and by Theorem \ref{thm:char} it is admissible. Slower divergences can also be admissible: $L(t)=\big[(T-t)\log\frac{eT}{T-t}\big]^{-1}$ has budget $\log\log\frac{eT}{T-t}\to\infty$, so no claim of a generally minimal growth order is made.
\end{corollary}

\begin{theorem}\label{thm:frontier}
Let $L\in L^1(0,T)$ with budget $\mathcal B$. For a closed loop $f$ admitting the envelope $L$, let $\kappa(f)=\sup_{z_0\ne0}\liminf_{t\to T^-}\|z(t;z_0)\|/\|z_0\|$ be its worst-case terminal contraction. Over the unrestricted class of closed-loop vector fields admitting $L$,
\[
\inf_{f}\ \kappa(f)\;=\;e^{-\mathcal B},
\]
and $f(t,z)=-L(t)z$ attains it.
\end{theorem}

\begin{proof}
For the lower bound, Theorem \ref{thm:necessity}(i) gives $\liminf_{t\to T^-}\|z(t)\|/\|z_0\|\ge e^{-\mathcal B}$ along every trajectory of every admissible $f$. For the upper bound, along $\dot z=-L(t)z$ every solution satisfies $\|z(t)\|/\|z_0\|=e^{-\int_0^tL}$.
\end{proof}

\begin{remark}[Universal floor, class-level attainment]\label{rem:plant}
The floor $e^{-\mathcal B}$ binds every trajectory of every plant, since Theorem \ref{thm:necessity} is a trajectory statement. The attainment is over the unrestricted vector-field class: on a specified plant with fixed structure, e.g.\ the double integrator $\dot z_1=z_2$, $\dot z_2=u$, the equation $\dot z_1=z_2$ is not assignable and the radial optimum need not be realizable by the available input. Characterizing the plant-constrained frontier is a separate, and in our view interesting, open question, which we record explicitly: \emph{for a given plant $\dot x=g(x,u)$ and envelope $L$, determine $\inf\kappa$ over the feedback laws whose closed loop admits $L$.} The universal floor $e^{-\mathcal B}$ bounds it from below for every plant; the numerical section quantifies the gap for one PT design of the double integrator.
\end{remark}

\begin{corollary}\label{cor:routes}
Let a trajectory with $z_0\ne0$ reach the origin by $T$ ($t_{\mathrm{hit}}\le T$) or be exactly PT-regulated at $T$. Then, by Lemma \ref{lem:gronwall}, every nonnegative $L$ satisfying $\|\dot z(t)\|\le L(t)\|z(t)\|$ a.e.\ \emph{along that trajectory} has $\int_0^TL=\infty$. Consequently: (a) if the closed loop admits a global gain envelope, that envelope has infinite integral, with Corollary \ref{cor:order} applying to it; (b) otherwise no global linear-growth bound holds, for instance because $\|f(t,z)\|/\|z\|$ is unbounded near the origin, which is how finite-time and fixed-time designs operate \cite{Haimo1986,BhatBernstein2000,Polyakov2012}. In particular, exact PT regulation with a bounded envelope is impossible; any bounded-envelope scheme, including frozen-gain implementations \cite{Shakouri2022,Zhou2020} and the constructions of \cite{SongSu2023} when they admit an envelope, delivers at best the accuracy $e^{-\mathcal B}$ of Theorem \ref{thm:frontier}.
\end{corollary}

\begin{corollary}[Per-decade law]\label{cor:decade}
By Theorem \ref{thm:necessity}(ii), reaching accuracies $10^{-1},\dots,10^{-N}$ requires budget at least $N\ln10$: each decade of accuracy costs a fixed increment of budget. The particular envelope $L(t)=c/(T-t)$ spends $\int_0^t c(T-s)^{-1}ds=c\log\frac{T}{T-t}$, so its budget grows logarithmically in the remaining time.
\end{corollary}

\begin{proposition}\label{thm:limit}
Under the hypotheses of Theorem \ref{thm:necessity} with $\mathcal B=\int_0^TL<\infty$, the limit $z(T):=\lim_{t\to T^-}z(t)$ exists, $z$ extends absolutely continuously to $[0,T]$, and $e^{-\mathcal B}\|z_0\|\le\|z(T)\|\le e^{\mathcal B}\|z_0\|$.
\end{proposition}

\begin{proof}
The standard Gr\"onwall bound gives $\|z(t)\|\le\|z_0\|e^{\mathcal B}$, hence $\|\dot z(t)\|\le L(t)\|z_0\|e^{\mathcal B}$ a.e., which is integrable; so $\dot z\in L^1(0,T;\R^n)$ and $z(t)=z_0+\int_0^t\dot z$ converges. Combine with Lemma \ref{lem:gronwall}.
\end{proof}

\section{Arrival costs beyond linear growth}\label{sec:osgood}

This section carries the floors beyond linear growth, through one comparison lemma covering all moduli at once. Let $\omega:(0,\infty)\to(0,\infty)$ be continuous, and consider AC trajectories obeying
\begin{equation}\label{eq:omegaenv}
\|\dot z(t)\|\ \le\ L(t)\,\omega(\|z(t)\|)\qquad\text{a.e.\ on }\{t: z(t)\neq0\},
\end{equation}
with $L\in L^1_{\mathrm{loc}}$ nonnegative. Define the \emph{arrival cost}
\[
\mathcal C(r)\ :=\ \int_0^{r}\frac{du}{\omega(u)}\ \in\ (0,\infty].
\]
The linear modulus gives $\mathcal C\equiv\infty$; the powers $u^{\alpha}$, $\alpha\in(0,1)$, give $\mathcal C(r)=r^{1-\alpha}/(1-\alpha)$.

\begin{theorem}\label{thm:osgood}
Let $z$ be AC on $[0,T)$ with $z(0)=z_0\neq0$ and satisfy \eqref{eq:omegaenv}, and write $r(t)=\|z(t)\|$.
(i) For every $t<t_{\mathrm{hit}}$,
\begin{equation}\label{eq:osgoodbound}
\int_{r(t)}^{r_0}\frac{du}{\omega(u)}\ \le\ \int_0^{t}L(s)\,ds .
\end{equation}
(ii) If $\inf_{t\in[0,T)}r(t)=0$, whether by $r(t)\to0$ as $t\to T^-$ or by arrival at some $t_{\mathrm{hit}}<T$, then $\int_0^TL\ \ge\ \mathcal C(r_0)$. Consequently
\begin{equation}\label{eq:hitfloor}
t_{\mathrm{hit}}\ge\inf\Big\{t:\int_0^tL\ge\mathcal C(r_0)\Big\}\ge\frac{\mathcal C(r_0)}{\bar L}\ \ \text{if }L\le\bar L.
\end{equation}
(iii) If moreover $\omega\in C([0,\infty))$ with $\omega(0)=0$ and $\omega>0$ on $(0,\infty)$, the bounds are exact: the inward radial field $f(t,z)=-L(t)\,\omega(\|z\|)\,z/\|z\|$ for $z\neq0$, extended by $0$ at the origin, is Carath\'eodory, produces $\dot r=-L(t)\omega(r)$, and arrives exactly when the budget first equals $\mathcal C(r_0)$; after arrival the origin is invariant for this field, so arrival is settling.
\end{theorem}

\begin{proof}
(i) On any $[0,\tau']$ with $\tau'<t_{\mathrm{hit}}$, $r$ is AC (reverse triangle inequality) with values in a compact subset of $(0,\infty)$, on which $1/\omega$ is continuous and bounded; so, by the chain rule for the $C^1$ function $\Psi(\rho)=\int_\rho^{r_0}du/\omega(u)$ composed with the AC function $r$,
\[
\frac{d}{dt}\int_{r(t)}^{r_0}\frac{du}{\omega(u)}\;=\;-\frac{\dot r(t)}{\omega(r(t))}\;\le\;\frac{\|\dot z(t)\|}{\omega(r(t))}\;\le\;L(t)
\]
a.e., using $\dot r\ge-\|\dot z\|\ge-L\,\omega(r)$; integration gives \eqref{eq:osgoodbound}. No monotonicity of $r$ is required. (ii) Pick $t_k$ with $r(t_k)\to0$, taking $t_k\uparrow t_{\mathrm{hit}}$ if arrival occurs; the left side of \eqref{eq:osgoodbound} increases to $\mathcal C(r_0)$, the right side never exceeds $\int_0^TL$, and \eqref{eq:hitfloor} follows since $\int_0^{t_{\mathrm{hit}}}L\ge\mathcal C(r_0)$. (iii) Continuity of the zero extension at the origin follows from $\|f(t,z)\|\le L(t)\,\omega(\|z\|)\to0$ as $z\to0$, so the field is Carath\'eodory. For it, $r$ satisfies $\dot r=-L(t)\omega(r)$ while $r>0$; the AC solution defined implicitly by $\int_{r(t)}^{r_0}du/\omega=\int_0^tL$ realizes it, reaches $0$ exactly when the budget reaches $\mathcal C(r_0)$, and the zero extension solves the equation thereafter.
\end{proof}

\begin{corollary}\label{cor:taxonomy}
The three convergence families of the literature are the regions of one integral.
(i) \emph{Impossibility beyond linear growth:} whenever $\int_{0^+}du/\omega=\infty$, no finite budget permits arrival or terminal convergence; the linear modulus is one such case, for which Theorem \ref{thm:osgood}(ii) reduces to Theorem \ref{thm:necessity}, and moduli violating linear growth can be divergent too, e.g.\ $\omega(u)=u\log(e/u)$ near zero; when $\mathcal C(r_0)<\infty$ it is the exact threshold, attained in Theorem \ref{thm:osgood}(iii).
(ii) \emph{Finite time:} for $\omega(u)=u^{\alpha}$, $\alpha<1$, and $L\equiv\bar L$, \eqref{eq:osgoodbound} integrates in closed form to $r(t)^{1-\alpha}\ge r_0^{1-\alpha}-\bar L(1-\alpha)t$, hence
\[
t_{\mathrm{hit}}\ \ge\ \frac{r_0^{1-\alpha}}{\bar L(1-\alpha)},
\]
with equality for the radial field $\dot z=-\bar L\|z\|^{\alpha-1}z$. The right side is the classical settling expression of finite-time control \cite{Haimo1986,BhatBernstein2000}, exact for the scalar design; the statement here is the reversed direction and quantifier: it is the minimum arrival time permitted by the speed constraint, for every admissible trajectory, regardless of feedback formula or dimension.
(iii) \emph{Fixed time:} a uniform arrival bound over all initial conditions is possible in the class exactly when $\int_0^{\infty}du/\omega<\infty$, since $\sup_{z_0\neq0}t_{\mathrm{hit}}(z_0)\ge\frac1{\bar L}\int_0^\infty\frac{du}{\omega(u)}$ by (ii) of Theorem \ref{thm:osgood}, and the radial field attains the bound. For the two-power modulus $\omega(u)=au^{p}+bu^{q}$ with $a,b>0$ and $0<p<q$, the integral converges at both ends precisely when $p<1<q$, recovering the structure of the fixed-time condition of \cite{Polyakov2012}. As a concrete instance, $\omega(u)=\sqrt u\,(1+u)$ gives $\mathcal C(r)=2\arctan\sqrt r\le\pi$, so no admissible trajectory arrives from any $z_0$ before $2\arctan\!\sqrt{\|z_0\|}\,/\bar L$, and none uniformly before $\pi/\bar L$.
\end{corollary}

\begin{corollary}[Uniform arrival]\label{cor:uniform}
Fix $\bar L>0$, with $\omega$ as in Theorem \ref{thm:osgood}(iii), and optimize over all loops of the class and all envelopes $0\le L(t)\le\bar L$ on a horizon containing the arrival times. A bound $T_u<\infty$ with $t_{\mathrm{hit}}(z_0)\le T_u$ for \emph{all} $z_0\neq0$ is achievable if and only if $\int_0^\infty du/\omega(u)<\infty$, and the least achievable $T_u$ equals $\frac1{\bar L}\int_0^\infty du/\omega(u)$: every admissible pair obeys $t_{\mathrm{hit}}(z_0)\ge\mathcal C(\|z_0\|)/\bar L$ with $\sup_{r}\mathcal C(r)=\int_0^\infty du/\omega$, and the radial field with $L\equiv\bar L$ attains these times exactly.
\end{corollary}

\begin{remark}[Relation to Osgood's uniqueness criterion]\label{rem:credit}
The integral $\int du/\omega$ is classical: divergence at the origin is Osgood's uniqueness criterion \cite{Osgood1898}; for a single scalar autonomous system the arrival time \emph{equals} the integral, a classical fact \cite{Haimo1986}; such integrals serve as sufficiency tools with settling upper estimates throughout finite-time theory \cite{MoulayPerruquetti2008,BhatBernstein2000}; exact settling integrals and least \emph{upper} bounds for given families are available, e.g.\ \cite{Aldana2019}; and the fixed-time constant of \cite{Polyakov2012} is the two-ended evaluation. What this section uses is the reversed direction and quantifier: the integral as a first-arrival floor for every trajectory of every closed loop in the modulus class, in every dimension, under time-varying budgets, with exact attainment inside the class; to the best of our knowledge, lower bounds of this class-wide kind are not stated in that literature.
\end{remark}

\begin{remark}[The budget clock]\label{rem:clock}
If $L>0$ a.e., the clock $\tau=\int_0^tL$ turns \eqref{eq:omegaenv} into $\|dz/d\tau\|\le\omega(\|z\|)$: measuring time by accumulated gain removes the envelope, and $\mathcal C(r_0)$ is the minimal $\tau$-time to the origin, which is why the cost is gain-free; intervals with $L=0$ force $\dot z=0$ and may be excised.
\end{remark}

\section{Linear time-varying loops}\label{sec:ltv}

For $\dot z=A(t)z$ the envelope $\|A(t)\|$ can be loose: the rotation $A=\omega J$, with $J$ skew-symmetric, has $\|A\|_2=\omega$ while $\|z(t)\|$ stays constant; the right one-sided rates are logarithmic norms. The next estimate is classical \cite{Coppel1965,Strom1975,Soderlind2006,Vrabel2020} and appears as Lemma 1 of \cite{Shakouri2022}; we include a proof to fix hypotheses.

\begin{lemma}\label{lem:mu}
Let $A\in L^1_{\mathrm{loc}}([0,T);\R^{n\times n})$ and let $z$ be an AC solution of $\dot z=A(t)z$ with $z_0\neq0$. Then for $t\in[0,T)$ and any $p$,
\begin{equation}\label{eq:mu2sided}
\|z_0\|_p\,e^{-\int_0^t\mu_p(-A(s))ds}\ \le\ \|z(t)\|_p\ \le\ \|z_0\|_p\,e^{\int_0^t\mu_p(A(s))ds}.
\end{equation}
\end{lemma}

\begin{proof}
$t\mapsto\mu_p(\pm A(t))$ is measurable, as a Lipschitz function of $A(t)$, and locally integrable since $|\mu_p(\pm A)|\le\|A\|_p$. Upper bound: $\|z\|_p>0$ everywhere by Lemma \ref{lem:gronwall} with $L=\|A(\cdot)\|_p$; at a.e.\ $r$, $z(r+h)=(I+hA(r))z(r)+o(h)$, so the AC function $\|z\|_p$ has a.e.\ derivative at most $\mu_p(A(r))\|z(r)\|_p$; integrate $\log\|z\|_p$. Lower bound: fix $t$, set $w(r)=z(t-r)$, note $\dot w(r)=-A(t-r)w(r)$, and apply the upper bound to $w$.
\end{proof}

\begin{theorem}\label{thm:ltv}
Let $A\in L^1_{\mathrm{loc}}([0,T);\R^{n\times n})$. If $\dot z=A(t)z$ is PT attractive at $T$, then for every $p$,
\begin{equation}\label{eq:muinf}
\int_0^t\mu_p(-A(s))\,ds\ \longrightarrow\ +\infty\qquad(t\to T^-).
\end{equation}
Conversely, $\int_0^t\mu_p(A(s))ds\to-\infty$ for some $p$ is sufficient, as in \cite{Shakouri2022}.
\end{theorem}

\begin{proof}
If \eqref{eq:muinf} fails there are $t_k\to T^-$ with $\int_0^{t_k}\mu_p(-A)\le c<\infty$, and the left bound in \eqref{eq:mu2sided} gives $\|z(t_k)\|_p\ge\|z_0\|_pe^{-c}$, so $z(t)\not\to0$. Sufficiency is the right bound in \eqref{eq:mu2sided}.
\end{proof}

\begin{corollary}\label{cor:pointwise}
Under the hypotheses of Theorem \ref{thm:ltv}, PT attractivity implies that for every $p$ and every $t_0<T$ the essential supremum of $\mu_p(-A(t))$ over $(t_0,T)$ is infinite, and hence so is that of $\|A(t)\|_p$: a bounded function on a finite interval has finite integral.
\end{corollary}

\begin{proposition}[Characterization of scalar and diagonal cases]\label{prop:scalar}
For $\dot x=a(t)x$ with $a\in L^1_{\mathrm{loc}}([0,T))$: the system is PT attractive at $T$ if and only if $\int_0^t a(s)ds\to-\infty$ as $t\to T^-$. For $A(t)=\mathrm{diag}(a_1,\dots,a_n)$: PT attractivity holds if and only if $\int_0^t a_i\to-\infty$ for every $i$. \emph{Proof:} $x(t)=x_0e^{\int_0^ta}$; the diagonal case decouples into $n$ scalar systems. In the scalar case the two conditions of Theorem \ref{thm:ltv} coincide, since $\mu(a)=a=-\mu(-a)$; the diagonal characterization is componentwise, not a coincidence of the matrix conditions, which involve $\max_i$ and $\min_i$ of the entries.
\end{proposition}

\begin{remark}[Relation to \cite{Shakouri2022}]\label{rem:shakouri}
The communiqu\'e proved, under the hypothesis that every entry of $A(t)$ has a possibly infinite limit as $t\to T^-$, the pointwise conclusions $\mu_p(-A(t))\to\infty$ and $\|A(t)\|_p\to\infty$, and its Remark 1 shows the hypothesis cannot simply be dropped from the limit form. Theorem \ref{thm:ltv} and Corollary \ref{cor:pointwise} follow directly from the same Lemma 1 of that paper and are elementary consequences of it; their value is the \emph{form}: the integral condition, and the essential-supremum conclusion, hold with no hypothesis beyond local integrability, survive the oscillating example of \cite{Shakouri2022}, and decide cases the pointwise test cannot, e.g.\ $\|A(t)\|=C(T-t)^{-0.6}$, which diverges pointwise yet has $\int_0^T\!C(T-s)^{-0.6}ds=\tfrac{C}{0.4}T^{0.4}<\infty$ and is thus excluded. The communiqu\'e also contains a sufficient condition and an exact, non-constructive characterization (its Theorem 2); the contribution here is not a missing direction there but the class-level theory of Sections \ref{sec:main} and \ref{sec:osgood}: what any closed loop, linear or not, can achieve under a gain constraint. Under its hypotheses the limit conclusion of \cite{Shakouri2022} is stronger than our essential supremum, Proposition \ref{prop:scalar} closes the scalar and diagonal cases, and for non-normal $A(t)$ a gap remains; closing it with a verifiable criterion is open.
\end{remark}

\begin{proposition}[Exact terminal error of gain freezing]\label{prop:frozen}
Let $A\in L^1_{\mathrm{loc}}([0,T];\R^{n\times n})$ follow any profile on $[0,t_s]$ and freeze thereafter, $A(t)=A_s:=A(t_s)$ for $t\in[t_s,T]$, and let $\Phi$ denote the transition matrix. Then $z(T)=Mz_0$ with
\[
M\;=\;e^{A_s(T-t_s)}\,\Phi(t_s,0),
\]
so in the Euclidean norm the terminal ratios $\|z(T)\|/\|z_0\|$ over $z_0\neq0$ fill exactly the interval $[\sigma_{\min}(M),\sigma_{\max}(M)]$, and
\[
\sigma_{\min}(M)\ \ge\ \exp\Big[-\!\int_0^{t_s}\!\|A(t)\|\,dt-(T-t_s)\|A_s\|\Big],
\]
so the budget floor of Theorem \ref{thm:necessity} is the universal bound beneath the exact spectrum.
\end{proposition}

\begin{proof}
Composition of transition matrices gives $M$; the singular-value description of the ratio set is standard. For the bound, $\sigma_{\min}(M)=1/\|M^{-1}\|$ and $\|M^{-1}\|\le\|\Phi(t_s,0)^{-1}\|\,\|e^{-A_s(T-t_s)}\|\le e^{\int_0^{t_s}\|A\|+(T-t_s)\|A_s\|}$ by the standard exponential bounds.
\end{proof}

\begin{remark}
The exact spectrum $[\sigma_{\min}(M),\sigma_{\max}(M)]$ is sharper than any envelope bound for the specified system; the floor is what survives for all systems and profiles. State-triggered freezing, with $t_s$ depending on $z_0$, needs a per-trajectory treatment, since $M$ then varies with the initial condition.
\end{remark}

\begin{remark}[The price of the freeze profile, at the envelope level]\label{rem:freezeprice}
Following the singular envelope $c/(T-t)$ until $t_s$ and freezing at the peak $\bar L=c/(T-t_s)$ spends $\int_0^TL=c\big(1+\log\frac{\bar LT}{c}\big)$, so the terminal ratio cannot go below $e^{-c}\,(c/\bar LT)^{c}$, polynomial in $1/\bar L$; a constant envelope of equal peak has floor $e^{-\bar LT}$, exponentially smaller. The singular segment buys uniformity of the transient, not terminal accuracy.
\end{remark}

\begin{remark}[Interpreting bounded-gain prescribed time]\label{rem:claims}
A bounded closed-loop envelope, $L\le\bar L$ with budget $\mathcal B\le\bar LT$, excludes exact arrival or terminal convergence by $T$ and bounds the achievable accuracy by the terminal ratio $e^{-\mathcal B}$; it guarantees nothing more by itself, and in particular no contraction, as $\dot z=z$ with envelope $L\equiv1$ shows, so practical regulation requires a separate performance guarantee on top of this limitation. Restricting the initial set does not remove the floor for the nonzero initial states to which the envelope applies. Finally, bounded controller coefficients need not produce a bounded closed-loop envelope, so claims of bounded-gain prescribed-time behavior, e.g.\ \cite{SongSu2023}, are to be read against the envelope the closed loop actually admits.
\end{remark}

\section{Numerical illustration}\label{sec:num}

Two experiments with $T=1$. First, the double integrator under $u=-k(t)x-c(t)\dot x$, $k=6/(T-t)^2$, $c=4/(T-t)$: the rest-to-rest solution $x(t)=3(T-t)^2-2(T-t)^3$ reaches the origin with $\dot x$ at $T$, consistent with Theorem \ref{thm:char} since this envelope is non-integrable. Saturating the gains, $T-t\mapsto\max(T-t,1/\nu)$ with $\nu\in\{1,1.5,2,3,4,5,6,8\}$, makes two budgets finite, and they must be distinguished: the speed-envelope budget $\mathcal B_L=\int_0^T\|A\|_2\,dt$ of Theorems \ref{thm:necessity} and \ref{thm:frontier}, and the logarithmic-norm budget $\mathcal B_\mu=\int_0^T\mu_2(-A)\,dt\le\mathcal B_L$, which bounds the radial rate rather than the speed and gives, by Lemma \ref{lem:mu}, the sharper floor $e^{-\mathcal B_\mu}\ge e^{-\mathcal B_L}$. Panel (a) plots against $\mathcal B_\mu$, which spans $[5.2,51.3]$; the achieved contraction spans $[0.605,0.085]$, above the floor $e^{-\mathcal B_\mu}$ at every $\nu$, and the radial design $\dot z=-\mu_2(-A(t))z$, whose ratio is exactly $e^{-\mathcal B_\mu}$, lies on it. To display the nonlinearity of the class, the same panel shows the member $\dot z=-c(\|z\|)\,\mu_2(-A(t))\,z$ with $c(r)=\tfrac12(1+\sin^23r)\in[\tfrac12,1]$: its ratio equals $\exp(-\int_0^Tc(\|z\|)\mu_2(-A)\,dt)$, hence lands in $[e^{-\mathcal B_\mu},e^{-\mathcal B_\mu/2}]$, verified at every $\nu$; the floor holds by the argument of Lemma \ref{lem:mu}, applied along its trajectories (Fig.~\ref{fig:final}a). The distance of the structured plant to this floor illustrates Remark \ref{rem:plant}.

Second, the arrival floors of Corollary \ref{cor:taxonomy} for $\omega(u)=\sqrt u\,(1+u)$, $\bar L=1$: 120 scalar loops $\dot x=-\bar L\,c(t)\,\omega(|x|)\,\mathrm{sign}(x)$, random piecewise-constant $c(t)\in[0.25,1]$, initial conditions over four orders of magnitude. Every recorded time lies above the floor $2\arctan\!\sqrt{\|z_0\|}\,/\bar L$, the radial design with $c\equiv1$ lies on it, and the uniform floor $\pi/\bar L$ caps the curve as $\|z_0\|$ grows (Fig.~\ref{fig:final}b). Arrival is detected at $|x|\le10^{-10}$, an under-estimate of $t_{\mathrm{hit}}$ by at most $2\sqrt{10^{-10}}/\bar L=2\times10^{-5}$; the smallest observed ratio to the floor, $1.0006$, is unaffected.

Third, the frozen-gain spectrum of Proposition \ref{prop:frozen} for the saturated design at $\nu=5$, freezing at $t_s=0.8$: integration of the transition matrix, converged over step sizes $10^{-4}$ to $2.5\times10^{-5}$ to six digits, gives $M=e^{A_s(T-t_s)}\Phi(t_s,0)$ with $\sigma_{\min}(M)=2.08\times10^{-4}$, $\sigma_{\max}(M)=0.141$, the exact range of terminal ratios of this system; $Mz_0$ agrees with direct integration to $10^{-14}$. The speed-envelope floor $e^{-\mathcal B_L}=9.7\times10^{-25}$, from Theorem \ref{thm:necessity}, sits far beneath it, as it must: the floor binds every plant and every profile, while the spectrum is the sharp answer for this one, which quantifies at once Proposition \ref{prop:frozen} and the plant-structure gap of Remark \ref{rem:plant}.

\begin{figure}[t]
\centering
\includegraphics[width=0.61\linewidth]{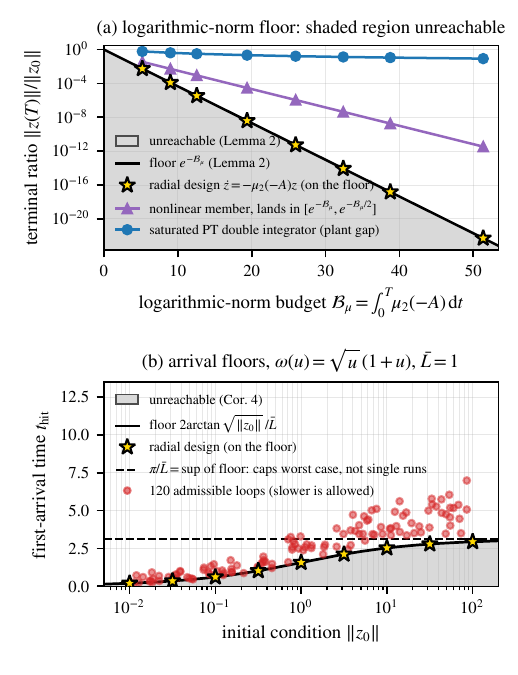}
\caption{Both bounds are floors: the shaded regions are unreachable, while slower behavior above them is always allowed. (a) Terminal ratios against the logarithmic-norm budget $\mathcal B_\mu$: the radial design (stars) sits exactly on the floor $e^{-\mathcal B_\mu}$ of Lemma \ref{lem:mu}, the nonlinear member lands in $[e^{-\mathcal B_\mu},e^{-\mathcal B_\mu/2}]$, and the structured double integrator sits far above, the plant-structure gap of Remark \ref{rem:plant}. (b) First-arrival times of 120 deliberately suboptimal loops ($c(t)\in[0.25,1]$) for the modulus $\sqrt u\,(1+u)$, $\bar L=1$: none enters the shaded region below $2\arctan\sqrt{\|z_0\|}/\bar L$, the radial design (stars) lies on its boundary, and $\pi/\bar L$ is the supremum of the floor, capping the best achievable worst case over all $z_0$, not individual runs.}
\label{fig:final}
\end{figure}

\section{Conclusion}

Time-constrained regulation has an exact price list, written by one comparison lemma. In the linear-growth class, a gain profile supports exact convergence at $T$ if and only if its integral diverges; a finite budget $\mathcal B$ buys at best the contraction $e^{-\mathcal B}$, universally over plants, and the trajectory then converges to a definite point away from the origin. Beyond linear growth, every modulus carries an arrival cost $\int_0^{r}du/\omega$: divergence at the origin makes arrival impossible under any finite budget, finite values are exact thresholds attained inside the class, the classical settling expressions become first-arrival floors, and the fixed-time condition is the two-ended convergence of the same integral, with uniform floor $\big(\int_0^\infty du/\omega\big)/\bar L$. For LTV loops, the integral and essential-supremum forms of the singularity necessity follow from a known two-sided estimate, and the terminal error of gain freezing is described exactly, with the budget floor beneath it. Four questions remain open. First, the plant-constrained frontier of Remark \ref{rem:plant}: the exact attainable contraction of an integrator chain under a given envelope is unknown. Second, a verifiable prescribed-time criterion for non-normal $A(t)$: the explicit integrals of Theorem \ref{thm:ltv} bracket the answer but do not meet. Third, limits in input energy, which Remark \ref{rem:effort} shows is a different quantity from the budget. Fourth, measurement noise: \cite{Aldana2023} suggests that outside the linear-growth class the binding limitation is noise amplification rather than gain.

\section*{Acknowledgment}
This paper is dedicated to the memory of Professor YangQuan Chen, who passed away during its preparation. His vision, generosity, and lifelong devotion to control theory shaped this work and its authors; it is submitted in his honor.

{\footnotesize
\bibliographystyle{IEEEtran}
\bibliography{references}

@article{SongKrstic2017,
  author  = {Song, Yongduan and Wang, Yujuan and Holloway, John and Krsti{\'c}, Miroslav},
  title   = {Time-varying feedback for regulation of normal-form nonlinear systems in prescribed finite time},
  journal = {Automatica},
  volume  = {83},
  pages   = {243--251},
  year    = {2017}
}

@article{HollowayKrstic2019,
  author  = {Holloway, John and Krsti{\'c}, Miroslav},
  title   = {Prescribed-time observers for linear systems in observer canonical form},
  journal = {IEEE Transactions on Automatic Control},
  volume  = {64},
  number  = {9},
  pages   = {3905--3912},
  year    = {2019}
}

@article{Krishnamurthy2020,
  author  = {Krishnamurthy, Prashanth and Khorrami, Farshad and Krsti{\'c}, Miroslav},
  title   = {A dynamic high-gain design for prescribed-time regulation of nonlinear systems},
  journal = {Automatica},
  volume  = {115},
  pages   = {108860},
  year    = {2020}
}

@article{ZhouShi2021,
  author  = {Zhou, Bin and Shi, Yang},
  title   = {Prescribed-time stabilization of a class of nonlinear systems by linear time-varying feedback},
  journal = {IEEE Transactions on Automatic Control},
  volume  = {66},
  number  = {12},
  pages   = {6123--6130},
  year    = {2021}
}

@article{Orlov2022,
  author  = {Orlov, Yury},
  title   = {Time space deformation approach to prescribed-time stabilization: {S}ynergy of time-varying and non-{L}ipschitz feedback designs},
  journal = {Automatica},
  volume  = {144},
  pages   = {110485},
  year    = {2022}
}

@article{OrlovCSL2022,
  author  = {Orlov, Yury and Verdes Kairuz, Ram{\'o}n I. and Aguilar, Luis T.},
  title   = {Prescribed-time robust differentiator design using finite varying gains},
  journal = {IEEE Control Systems Letters},
  volume  = {6},
  pages   = {620--625},
  year    = {2022}
}

@article{Espitia2022,
  author  = {Espitia, Nicol{\'a}s and Steeves, Drew and Perruquetti, Wilfrid and Krsti{\'c}, Miroslav},
  title   = {Sensor delay-compensated prescribed-time observer for {LTI} systems},
  journal = {Automatica},
  volume  = {135},
  pages   = {110005},
  year    = {2022}
}

@article{Ye2024,
  author  = {Ye, Hefu and Wen, Changyun and Song, Yongduan},
  title   = {Decentralized prescribed-time control for interconnected nonlinear systems via output-feedback},
  journal = {Automatica},
  pages   = {111571},
  year    = {2024}
}

@article{SongSurvey2023,
  author  = {Song, Yongduan and Ye, Hefu and Lewis, Frank L.},
  title   = {Prescribed-time control and its latest developments},
  journal = {IEEE Transactions on Systems, Man, and Cybernetics: Systems},
  volume  = {53},
  year    = {2023}
}

@article{SongSu2023,
  author  = {Song, Yongduan and Su, Jie},
  title   = {A unified {L}yapunov characterization for finite time control and prescribed time control},
  journal = {International Journal of Robust and Nonlinear Control},
  volume  = {33},
  number  = {4},
  pages   = {2930--2949},
  year    = {2023}
}

@article{Shakouri2022,
  author  = {Shakouri, Amir},
  title   = {On the prescribed-time attractivity and frozen-time eigenvalues of linear time-varying systems},
  journal = {Automatica},
  volume  = {140},
  pages   = {110173},
  year    = {2022}
}

@article{Zhou2020,
  author  = {Zhou, Bin},
  title   = {Finite-time stability analysis and stabilization by bounded linear time-varying feedback},
  journal = {Automatica},
  volume  = {121},
  pages   = {109191},
  year    = {2020}
}

@article{Aldana2023,
  author  = {Aldana-L{\'o}pez, Rodrigo and Seeber, Richard and Haimovich, Hernan and G{\'o}mez-Guti{\'e}rrez, David},
  title   = {On inherent limitations in robustness and performance for a class of prescribed-time algorithms},
  journal = {Automatica},
  volume  = {158},
  pages   = {111284},
  year    = {2023}
}

@article{Aldana2019,
  author={Aldana-L{\'o}pez, Rodrigo and G{\'o}mez-Guti{\'e}rrez, David and Jim{\'e}nez-Rodr{\'\i}guez, Esteban and S{\'a}nchez-Torres, Juan Diego and Defoort, Michael},
  title   = {Enhancing the settling time estimation of a class of fixed-time stable systems},
  journal = {International Journal of Robust and Nonlinear Control},
  volume  = {29},
  pages={4135--4148},
  year    = {2019}
}

@article{Haimo1986,
  author  = {Haimo, V. T.},
  title   = {Finite time controllers},
  journal = {SIAM Journal on Control and Optimization},
  volume  = {24},
  number  = {4},
  pages   = {760--770},
  year    = {1986}
}

@article{BhatBernstein2000,
  author  = {Bhat, Sanjay P. and Bernstein, Dennis S.},
  title   = {Finite-time stability of continuous autonomous systems},
  journal = {SIAM Journal on Control and Optimization},
  volume  = {38},
  number  = {3},
  pages   = {751--766},
  year    = {2000}
}

@article{MoulayPerruquetti2008,
  author  = {Moulay, Emmanuel and Perruquetti, Wilfrid},
  title   = {Finite time stability conditions for non-autonomous continuous systems},
  journal = {International Journal of Control},
  volume  = {81},
  number  = {5},
  pages   = {797--803},
  year    = {2008}
}

@article{Osgood1898,
  author  = {Osgood, William F.},
  title   = {Beweis der {E}xistenz einer {L}{\"o}sung der {D}ifferentialgleichung $dy/dx=f(x,y)$ ohne {H}inzunahme der {C}auchy--{L}ipschitz'schen {B}edingung},
  journal = {Monatshefte f{\"u}r Mathematik und Physik},
  volume  = {9},
  pages   = {331--345},
  year    = {1898}
}

@article{Polyakov2012,
  author  = {Polyakov, Andrey},
  title   = {Nonlinear feedback design for fixed-time stabilization of linear control systems},
  journal = {IEEE Transactions on Automatic Control},
  volume  = {57},
  number  = {8},
  pages   = {2106--2110},
  year    = {2012}
}

@article{Vrabel2020,
  author  = {Vr{\'a}bel, R{\'o}bert},
  title   = {A note on uniform exponential stability of linear periodic time-varying systems},
  journal = {IEEE Transactions on Automatic Control},
  volume  = {65},
  number  = {4},
  pages   = {1647--1651},
  year    = {2020}
}

@book{Coppel1965,
  author    = {Coppel, W. A.},
  title     = {Stability and Asymptotic Behavior of Differential Equations},
  publisher = {D. C. Heath},
  address   = {Boston, MA},
  year      = {1965}
}

@article{Strom1975,
  author  = {Str{\"o}m, Torsten},
  title   = {On logarithmic norms},
  journal = {SIAM Journal on Numerical Analysis},
  volume  = {12},
  number  = {5},
  pages   = {741--753},
  year    = {1975}
}

@article{Soderlind2006,
  author  = {S{\"o}derlind, Gustaf},
  title   = {The logarithmic norm. {H}istory and modern theory},
  journal = {BIT Numerical Mathematics},
  volume  = {46},
  pages   = {631--652},
  year    = {2006}
}

@article{ZhangZhouDuan2024,
  author  = {Zhang, Kang-Kang and Zhou, Bin and Duan, Guang-Ren},
  title   = {Global prescribed-time output feedback control of a class of uncertain nonlinear systems by linear time-varying feedback},
  journal = {Automatica},
  volume  = {165},
  pages   = {111680},
  year    = {2024}
}

@article{LiMaTian2024,
  author  = {Li, Yangang and Ma, Ruicheng and Tian, Xiaomei},
  title   = {Prescribed-time tracking with prescribed performance for a class of strict-feedback nonlinear systems},
  journal = {Journal of the Franklin Institute},
  volume  = {361},
  pages   = {106638},
  year    = {2024}
}
}

\end{document}